\pdfoutput=1
\documentclass[a4paper,12pt]{amsart}

\usepackage[utf8]{inputenc}
\usepackage[T1]{fontenc}
\usepackage{microtype}

\usepackage[cmyk]{xcolor}
\usepackage{tikz}

\usepackage{amssymb,amsmath,mathtools}
\usepackage{amsthm}
\usepackage{bm}

\usepackage[unicode]{hyperref}

\hypersetup{
  pdftitle={The Ordered Join of Impartial Games},
  pdfauthor={Mišo Gavrilović and Alexander Thumm},
  pdfsubject={The Ordered Join of Impartial Games},
  linktoc=all,
  hidelinks,
}

\theoremstyle{plain}
\newtheorem{theorem}{Theorem}[section]
\newtheorem{lemma}[theorem]{Lemma}
\newtheorem{question}[theorem]{Question}

\theoremstyle{definition}
\newtheorem{definition}[theorem]{Definition}

\makeatletter
\providecommand\given{}
\newcommand\set@symbol[1][]{\nonscript\,#1\vert\allowbreak\nonscript\,\mathopen{}}
\DeclarePairedDelimiterX\set[1]\{\}{\renewcommand\given{\set@symbol[\delimsize]}#1}
\makeatother

\DeclarePairedDelimiter\card\lvert\rvert

\DeclareMathOperator{\mex}{mex}   
\DeclareMathOperator{\ex}{ex}     
\newcommand{\nadd}{\oplus}        

\newcommand*{\mcomp}{\ltimes}     

\newcommand*{\nim}{\textsc{Nim}}
\newcommand*{\ghack}{\textsc{Green Hackenbush}}

\newcommand*{\game}[1]{\bm{#1}}

\newcommand*{\gtype}[1][\infty]{\Gamma_{#1}}
\newcommand*{\gnum}{\gtype[0]}
\newcommand*{\gset}{\gtype[1]}

\newcommand{\dsum}{+}             
\newcommand{\osum}{:}             
\newcommand{\ojoin}{\ltimes}      

\newcommand{\PSPACE}{\ensuremath{{\text{\small\textsf{PSPACE}}}}}

\begin{document}

\title{The Ordered Join of Impartial Games}

\author{Mi\v{s}o Gavrilovi\'{c}}
\email{miso.gavrilovic@gmx.de}
\author{Alexander Thumm}
\email{alexander.thumm@gmx.net}


\begin{abstract}
  Inspired by the theory of poset games, we introduce a new compound of impartial combinatorial games and provide a complete analysis in the spirit of the Sprague-Grundy theory.

  Furthermore, we establish several substitution and reduction principles for this compound and consider its computational aspects.
\end{abstract}

\maketitle

\section{Introduction}\label{sec:introduction}

Combinatorial game theory studies two-player games of no chance and no hidden information.
Its roots lie in Bouton's \cite{B1901} complete analysis of the game \nim{}, in which he provided an efficient arithmetic method to determine the outcome of the game's positions.
Inspired by this, Sprague \cite{S1935} and Grundy \cite{G1939} independently found a proper notion, the (Sprague-)Grundy number, to expand Bouton's groundbreaking ideas to arbitrary impartial games.

The Grundy number determines the outcome of a game's position and works nicely in tandem with the disjunctive sum of games.
Furthermore, the Sprague-Grundy Theorem gives a characterization of Grundy numbers as outcome classes under disjunctive sums, which, in turn, is used as the center piece of modern combinatorial game theory.

Our work presented in this article is of similar nature, studying the Grundy set and a new compound, leading to an analogous characterization.
The compound itself, the \emph{ordered join}, is a direct generalization of the disjunctive and ordinal sums of games, combining arbitrary games through the structure of a poset.
As such, it has many similarities to poset games, a type of games introduced by Gale and Neyman \cite{GN1982}.

\medbreak

After introducing the ordered join, we establish two substitution principles, culminating in the main result of this article.
Borrowing from the theory of poset games, we then expand a symmetry result of Fenner and Rogers \cite{FR2015} to ordered joins.
We finally consider some computational aspects of the presented theory and, to conclude, discuss some prospects.

\medbreak

Lastly, the authors would like to mention a related result for poset games \cite{CF2021}, which appeared online during the writing of this article.

\section{Preliminaries}\label{sec:preliminaries}

For an introduction to combinatorial game theory we refer the reader to common literature (e.g.\ \cite{BCG2001,S2013,C1976}).

In this article we only consider \emph{impartial combinatorial games} with \emph{normal-play rules}.
We call such a game \emph{terminating} or \emph{loopfree} if it admits no infinite sequence of moves and \emph{short} if additionally there are only finitely many game positions.
An \emph{option} of a game $G$ is a game $G'$ which is the result of a single move made in $G$.
We denote this relationship by $G \to G'$.
Furthermore, we sometimes lazily write $G'$ to denote any typical option of $G$.
Lastly, we write $o(G)$ for the \emph{outcome class} of a terminating game $G$.

\medbreak

The remainder of this section is concerned with two well-known constructions on games, the disjunctive sum and the ordinal sum.

\subsection{The Disjunctive Sum}
The first compound we discuss is the disjunctive sum, a cornerstone of combinatorial game theory.

\begin{definition}
  The \emph{disjunctive sum} $G \dsum H$ of the games $G$ and $H$ is the game with options $G' \dsum H$ and $G \dsum H'$ for all $G'$ and $H'$.
\end{definition}

In other words, on a player's turn he/she chooses a component $G$ or $H$ and makes a move in it, leaving the other as it is.

The question for the outcome of a disjunctive sum is answered by the Sprague-Grundy Theorem.
It is the condensed result of the independent works of Sprague \cite{S1935} and Grundy \cite{G1939}, extending upon the earlier work of Bouton \cite{B1901} on \nim.
Essential to this theory is the (\emph{Sprague-})\emph{Grundy number} or \emph{nimber}.

\begin{definition}
  The \emph{Grundy number} of a terminating game $G$ is
  \[ 
    \gnum(G) \coloneqq \mex{} \set{\gnum(G') \given G \to G' } 
  \]
  where $\mex \Lambda$ is the \emph{minimal excluded} ordinal of the set of ordinals $\Lambda$.
\end{definition}

For example, the Grundy number of the empty game $\game{0}$ is $\gnum(\game{0}) = 0$ and that of a \nim{} stack of height $\lambda$, written as $\game{\lambda}$ or $*\lambda$, is $\gnum(\game{\lambda}) = \lambda$.

We say $G$ and $H$ are \emph{0-equivalent}, written $G \simeq_0 H$, if $\gnum(G) = \gnum(H)$.
Contrary to the common notation, we reserve $G = H$ to say that $G$ and $H$ are \emph{exactly} the same game.

\medbreak

The Sprague-Grundy Theorem has many portrayals and interpretations, but for our purpose we focus on the following aspect.

\begin{theorem}\label{thm:sprague-grundy}
  Let $G$ and $H$ be terminating games.
  Then $G \simeq_0 H$ if and only if $o(G \dsum X) = o(H \dsum X)$ for all terminating games $X$.
\end{theorem}

The Grundy number of the disjunctive sum of arbitrary games depends only on the Grundy numbers of its components.
Indeed,
\[ 
  \gnum(G \dsum H) = \gnum(G) \nadd \gnum(H). 
\]
Herein $\nadd$ denotes the \emph{nimber addition} of ordinals, which in the finite case can be calculated as the bitwise \texttt{xor} of binary representations.

This result gives meaning to the Grundy number as \emph{the} key invariant for the disjunctive sum, as analysing disjunctive sums is strongly connected to analysing Grundy numbers.

\subsection{The Ordinal Sum}
The second compound we discuss is the ordinal sum, a common construction for poset games and an important tool for the analysis of \ghack.

\begin{definition}
  The \emph{ordinal sum} $G \osum H$ of the games $G$ and $H$ is the game with options $G \osum H'$ and $G' \osum \game{0}$ for all $G'$ and $H'$.
\end{definition}

In other words, on a player's turn he/she chooses a component $G$ or $H$ and makes a move in it, leaving $G$ as it is if moving in $H$ but discarding $H$ if moving in $G$.

The Grundy number of an ordinal sum is not completely determined by the Grundy numbers of its components, e.g.\ $\game{0} \osum \game{1} \not\simeq_0 (\game{1} \dsum \game{1}) \osum \game{1}$.
However, it is sufficient to look one move into the future.

\begin{definition}
  The \emph{Grundy set} $\gset(G)$ of a terminating game $G$ is
  \[ 
    \gset(G) \coloneqq \set{\gnum(G') \given G \to G'}. 
  \]
\end{definition}

In other words, it is the set of Grundy numbers of the options of $G$, and thus we have the relation $\gnum(G) \coloneqq \mex\gset(G)$.

We say $G$ and $H$ are \emph{1-equivalent}, written $G \simeq_1 H$, if $\gset(G) = \gset(H)$.
Using this notation, we state the \emph{colon principle} and additional related results, cf.\ \cite[pp.~190, 219f.]{BCG2001}.

\begin{theorem}\label{thm:ordinal_substitution}
  Let $G,H$ and $X$ be terminating games.
  \begin{enumerate}
    \item If $G \simeq_0 H$, then $X \osum G \simeq_0 X \osum H$.
    \item If $G \simeq_1 H$, then $G \osum X \simeq_1 H \osum X$.
  \end{enumerate}
\end{theorem}

Furthermore, as in the case of disjunctive sums, there exists a simple formula for the Grundy number of an ordinal sum:
\[ 
  \gnum(G \osum H) = \ex_{\gnum(H)} \gset(G).
\]
Herein $\ex_\rho \Lambda \coloneqq \mex \left(\Lambda \cup \set{ \ex_\mu \Lambda \given \mu < \rho }\right)$ is the \emph{$\rho$-th excluded} ordinal of the set of ordinals $\Lambda$, in particular $\ex_0 \Lambda = \mex \Lambda$.

\section{The Ordered Join}\label{sec:ordered_join}

In this section we will introduce the main notion --- the \emph{ordered join} of games.
The general idea is to combine arbitrary games through the structure of a poset.
We then show how the ordered join generalizes the disjunctive and ordinal sums and study its compositional behaviour.

\medbreak

An ordered join is a game consisting of two layers of play.
On the lower layer we have a family of independent games, in which ultimately the moves will be made.
On the upper layer we have a poset, which connects the games of the lower layer and does not change when making a move.

\begin{definition}\label{def:ordered_join}
  Let $S$ be a poset and $(G_i)_{i\in S}$ a family of games.
  We denote by $S \ojoin_i G_i$ the \emph{ordered join} of the games $G_i$ in the shape $S$.
  It is the game whose options are all ordered joins $S \ojoin_i G'_i$ satisfying $G_{i_0} \to G'_{i_0}$ for one $i_0 \in S$, with $G'_i = \game{0}$ for $i > i_0$ and $G_i = G'_i$ otherwise.
\end{definition}

In other words, on a player's turn he/she chooses a component $G_{i_0}$ and makes a move in it, discarding all components $G_i$ with $i > i_0$ and leaving others as they are.

Note that the definition of the ordered join does allow for a component to be the empty game $\game{0}$, which has technical advantages but isn't affecting the game.
However, it will also be useful to have a notion for the essential part of the compound.
Motivated by this, we define the \emph{support} of an ordered join $S \ojoin_i G_i$ to be the induced subposet $T \subseteq S$ of its shape consisting of all $i \in S$ with $G_i \neq \game{0}$.

\begin{figure}[ht]
  \begin{tikzpicture}
    \tikzset{
        component/.style={circle, inner sep=0, minimum size=20pt},
        pics/ordered join/.style n args={6}{code={
          \node[component] (L0C) at (   0, -2) {\small #1};
          \node[component] (L1L) at (-1.5,  0) {\small #2};
          \node[component] (L1R) at (+1.5,  0) {\small #3};
          \node[component] (L2L) at (  -3, +2) {\small #4};
          \node[component] (L2C) at (   0, +2) {\small #5};
          \node[component] (L2R) at (  +3, +2) {\small #6};

          \draw[thick] (L0C) -- (L1L)  (L0C) -- (L1R);
          \draw[thick] (L1L) -- (L2L)  (L1L) -- (L2C);
          \draw[thick] (L1R) -- (L2C)  (L1R) -- (L2R);
        }}
    }

    \pic[scale=0.4] (G)   at (-2,0) 
      {ordered join={$G_1$}{$G_2$}{$G_3$}{$G_4$}{$G_5$}{$G_6$}};

    \node at (0,0) {$\longrightarrow$};

    \pic[scale=0.4] (G')  at (+2,0) 
      {ordered join={$G_1$}{$G_2$}{$G'_3$}{$G_4$}{$\game{0}$}{$\game{0}$}};
  \end{tikzpicture}
  \caption{A move in an ordered join of games.}
  \label{fig:ordered_join}
\end{figure}
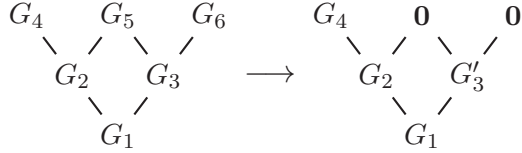

\medbreak

The question whether or not an ordered join is a terminating game leads to the following condition on its shape.
This condition appears in the same role within the theory of poset games.

\begin{definition}
  A poset $S$ is called a \emph{well partially ordered set} (\emph{wposet}) if every infinite sequence in $S$ contains an infinite ascending subsequence.
\end{definition}

Equivalently, $S$ is a wposet if it contains neither an infinitely descending chain nor an infinite antichain (see \cite[Theorem 2.1]{H1952}).
Using both, one readily makes the following observation.

\begin{lemma}
  Let $G = S \ojoin_i G_i$ be an ordered join.
  \begin{enumerate}
    \item The game $G$ is terminating if and only if the support of $G$ is a wposet and each component $G_i$ is terminating.
    \item The game $G$ is short if and only if the support of $G$ is finite and each component $G_i$ is short.
  \end{enumerate}
\end{lemma}

The similarity in the terminating conditions of poset games and ordered joins is obviously no coincidence, as the two theories are closely related.
For example, the ordered join $S \ojoin_i \game{1}$ is the same game as the poset game on $S$.
Therefore, each poset game has a natural description as a nontrivial ordered join.

\medbreak

The flexibility of the ordered join lies in its shape.
Specific shapes result in disjunctive and ordinal sums of games.
The disjunctive sum $G_1 \dsum G_2$ arises as the ordered join $A_2 \ojoin_i G_i$, where $A_2 = \set{1,2}$ is the two-element antichain.
Furthermore,
\[
  (S_1 + S_2) \ojoin_i G_i = (S_1 \ojoin_i G_i) + (S_2 \ojoin_i G_i)
\]
where $S_1 + S_2$ is the disjoint sum of the shapes $S_1$ and $S_2$.
On the other hand, the ordinal sum $G_1 \osum G_2$ arises as the ordered join $C_2 \ojoin_i G_i$, where $C_2 = \set{1 < 2}$ is the two-element chain, and a similar equation for these compounds and the ordinal sum of posets holds.

In combination, the disjunctive and ordinal sum lead to ordered joins in the shape of \emph{series-parallel posets}.
In contrast, the ordered join allows for more general interactions between its components.
For example, the ordered join $N \ojoin_i \game{1}$ in the shape $N \coloneqq \set{1 < 2 > 3 < 4}$ is neither a disjunctive sum nor an ordinal sum of games in a nontrivial way.

As discussed above, ordered joins are closed under formation of disjunctive and ordinal sums.
In the same way, they are also closed under ordered joins themselves.
Formalizing this leads us to a related construction from the theory of posets: the 
\emph{modular composition} (also \emph{substitution}, \emph{ordinal sum}, or \emph{X-join}; see e.g.\ \cite{MR1984,S1959,S1971}).

\begin{definition}\label{def:modular_composition}
Let $R$ be a poset and let $(S_i)_{i \in R}$ be a family of posets. 
The \emph{modular composition} $R \mcomp_i S_i$ is the disjoint union $\bigsqcup_i S_i$ with the inherited order extended in such a way that $j_1 \leq j_2$ holds for all elements $j_1 \in S_{i_1}$ and $j_2 \in S_{i_2}$ with $i_1 \leq i_2$.
\end{definition}

In other words, the modular composition is the poset obtained by replacement of the elements $i \in R$ by their associated posets $S_i$.
Examples of it include the disjoint sum, the ordinal sum, and the lexicographic product of posets.
The following immediate observation summarizes the relation between the modular composition and the ordered join, generalizing the above equations.

\begin{lemma}\label{lem:composition}
  $(R \mcomp_i S_i) \ojoin_j G_j = R \ojoin_i (S_i \ojoin_j G_j)$.
\end{lemma}

\section{Substitution Principles}\label{sec:substitution}

In this section we establish the claimed analogue of the Sprague-Grundy Theorem for the ordered join.
It is given in the form of two theorems.
The first is a general substitution principle.
It shows that the Grundy set and thus also the outcome class of an ordered join, given its shape, is completely determined by the Grundy sets of its components.
The second concerns outcome classes of games under ordered joins.

While the former has apparent theoretical and practical utility, the consequences of the latter are of a more philosophical nature.
We will discuss some of these once both results have been presented.

\subsection{General Substitution Theorem}\label{ssec:general_substitution}

As a preliminary to the main result, we first establish a related substitution principle.
Despite its auxiliary nature, this special substitution theorem is a valuable tool for computations.
It is a simple generalization of the substitution of a game $H$ within an ordinal sum $G \osum H$ as in Theorem~\ref{thm:ordinal_substitution}.

\begin{theorem}\label{thm:special_substitution}
  Let $S \ojoin_i G_i$ and $S \ojoin_i H_i$ be terminating ordered joins in the same shape $S$.
  Suppose that $i_0 \in S$ is maximal, and that $G_i = H_i$ for all $i \neq i_0$.
  If $G_{i_0} \simeq_0 H_{i_0}$, then $S \ojoin_i G_i \simeq_0 S \ojoin_i H_i$.
\end{theorem}
\begin{proof}
  Let $G \coloneqq S \ojoin_i G_i$ and $H \coloneqq S \ojoin_i H_i$.
  We show $G \dsum H \simeq_0 \game{0}$, which is equivalent to $ G \simeq_0 H$.
  By symmetry, it suffices to show that each move $G \to G'$ has an answer $H \to H'$ with $G' \simeq_0 H'$.  
  Consider such a move $G \to G'$.

  If the move is not in $i_0$, we can copy the move in $H$ and are done by induction.
  Otherwise, the move is in $i_0$, and thus the Grundy number of the component in $i_0$ changes.
  It either increases or decreases.
  \begin{itemize}
    \item If $\gnum(G_{i_0}) < \gnum(G'_{i_0})$, there is a move $G'_{i_0} \to G''_{i_0}$ with $G_{i_0} \simeq_0 G''_{i_0}$.
    \item If $\gnum(G_{i_0}) > \gnum(G'_{i_0})$, there is a move $H_{i_0} \to H'_{i_0}$ with $G'_{i_0} \simeq_0 H'_{i_0}$.
  \end{itemize}
  In both cases the claim follows by induction.
\end{proof}

Recall that empty components in an ordered join do not affect the game.
Thus, the theorem is applicable even if $i_0$ is not maximal in the shape but maximal in the support of the ordered join.

Furthermore, the converse of the theorem also holds.
Indeed, if $G_{i_0} \not\simeq_0 H_{i_0}$, say $\gnum(G_{i_0}) < \gnum(H_{i_0})$, then there is a move $H_{i_0} \to H'_{i_0}$ with $G_{i_0} \simeq_0 H'_{i_0}$.
By Theorem~\ref{thm:special_substitution}, $S \ojoin_i G_i \simeq_0 S \ojoin_i H'_i \not\simeq_0 S \ojoin_i H_i$.

\medbreak

By contrast, our general substitution theorem allows for replacement of arbitrary components while preserving the Grundy set of the compound.
As such, it generalizes substitution of $G$ within $G \osum H$.

\begin{theorem}\label{thm:general_substitution}
  Let $S \ojoin_i G_i$ and $S \ojoin_i H_i$ be terminating ordered joins in the same shape $S$.
  If $G_i\simeq_1 H_i$ for each $i \in S$, then $S \ojoin_i G_i \simeq_1 S \ojoin_i H_i$.
\end{theorem}
\begin{proof}
  By symmetry, it suffices to show that $\gset(G) \subseteq \gset(H)$, where $G \coloneqq S \ojoin_i G_i$ and $H \coloneqq S \ojoin_i H_i$.
  Consider a move $G\to G'$, say in $i_0$.
  Since $G_{i_0} \simeq_1 H_{i_0}$, we find a move $H \to H'$ in $i_0$ with $G'_{i_0} \simeq_0 H'_{i_0}$.
  
  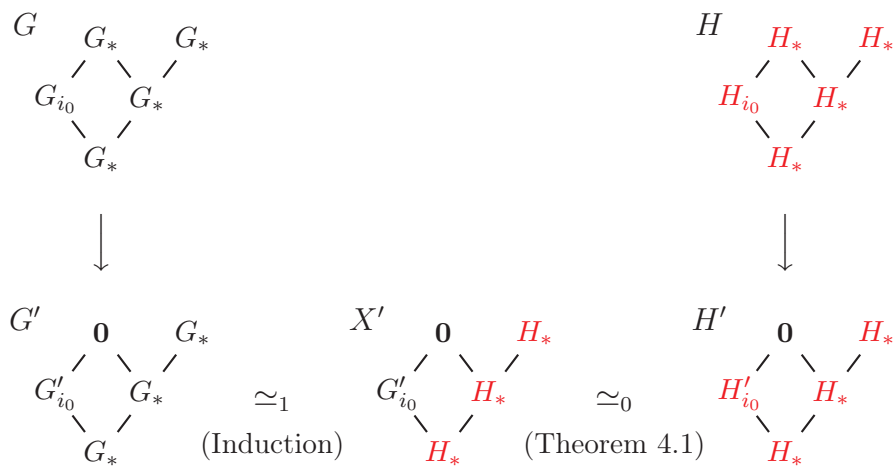
\begin{figure}[ht]
    \begin{tikzpicture}
      \tikzset{
        component/.style={circle, inner sep=0, minimum size=20pt},
        pics/substitution/.style n args={4}{code={
          \node (NAME) at (-2.5, +2.5) {#1};

          \node[component] (L0C) at (   0, -2) {\small #4};
          \node[component] (L1L) at (-1.5,  0) {\small #2};
          \node[component] (L1R) at (+1.5,  0) {\small #4};
          \node[component] (L2C) at (   0, +2) {\small #3};
          \node[component] (L2R) at (  +3, +2) {\small #4};

          \draw[thick] (L0C) -- (L1L)  (L0C) -- (L1R);
          \draw[thick] (L1L) -- (L2C);
          \draw[thick] (L1R) -- (L2C)  (L1R) -- (L2R);
        }}    
      }

      \newcommand{\gc}{\color{black}}
      \newcommand{\hc}{\color{red}}

      \pic[scale=0.4] (G) at (-4.5, 1.9) 
        {substitution={$G$}{$\gc G_{i_0}$}{$\gc G_\ast$}{\gc $G_\ast$}};
      \pic[scale=0.4] (H) at (+4.5, 1.9)
        {substitution={$H$}{$\hc H_{i_0}$}{$\hc H_\ast$}{$\hc H_\ast$}};
      
      \node at (-4.5, 0) {$\Big\downarrow$};
      \node at (+4.5, 0) {$\Big\downarrow$};

      \pic[scale=0.4] (G') at (-4.5, -2) 
        {substitution={$G'$}{$\gc G'_{i_0}$}{$\game{0}$}{$\gc G_\ast$}};
      
      \node at (-2.25, -2.1) {$\simeq_1$};
      \node at (-2.25, -2.7) {\small (Induction)};

      \pic[scale=0.4] (X') at (   0, -2) 
        {substitution={$X'$}{$\gc G'_{i_0}$}{$\game{0}$}{$\hc H_\ast$}};
      
      \node at (+2.25, -2.1) {$\simeq_0$};
      \node at (+2.25, -2.7) {\small (Theorem~\ref{thm:special_substitution})};
      
      \pic[scale=0.4] (H') at (+4.5, -2)
        {substitution={$H'$}{$\hc H'_{i_0}$}{$\game{0}$}{$\hc H_\ast$}};
    \end{tikzpicture}
    \caption{The proof of Theorem~\ref{thm:general_substitution}.}
    \label{fig:general_substitution}
  \end{figure}
  
  We claim $G' \simeq_0 H'$.
  In order to prove this, let us consider the mixed game $X' \coloneqq S \ojoin_i X'_i$ (see Figure~\ref{fig:general_substitution}), where $X'_i = H_i$ for $i \neq i_0$ and $X'_{i_0} = G'_{i_0}$.
  We obtain $X'$ from $G'$ by replacing every component in $G'$ except $G'_{i_0}$.
  By induction, we have $G' \simeq_1 X'$.
  Because $X'_{i_0} = G'_{i_0} \simeq_0 H'_{i_0}$ and $i_0$ is maximal in the support of $X'$, Theorem~\ref{thm:special_substitution} yields $X' \simeq_0 H'$.

  Since $G'$ was arbitrary, this proves the inclusion $\gset(G) \subseteq \gset(H)$.
\end{proof}

\subsection{Outcome-Equivalence Theorem}

The Sprague-Grundy Theorem (Theorem~\ref{thm:sprague-grundy}) characterizes the Grundy numbers as outcome classes of games under disjunctive sums.
In a similar fashion, the following theorem provides a characterization of Grundy sets as outcome classes of games under ordered joins.

\begin{theorem}\label{thm:outcomes}
  Let $G$ and $H$ be terminating games.
  Then $G \simeq_1 H$ if and only if $o(S \ojoin_i G_i) = o(S \ojoin_i H_i)$ for all terminating ordered joins with $G_{i_0} = G$ and $H_{i_0} = H$ for some $i_0 \in S$ and $G_i = H_i$ otherwise.
\end{theorem}
\begin{proof}
  If $G \simeq_1 H$, then, by Theorem~\ref{thm:general_substitution}, any two ordered joins as above have the same Grundy set and hence the same outcome class.
  
  For the converse, consider the games $\game{\lambda} \dsum (G \osum \game{\rho})$, where $\lambda$ and $\rho$ range over the ordinal numbers.
  The Grundy number of such a game is 
  \[
    \gnum(\game{\lambda} \dsum (G \osum \game{\rho})) = \lambda \nadd \ex_\rho \gset(G).
  \]
  Clearly, the latter is zero if and only if $\lambda = \ex_\rho \gset(G)$.
  Thus, the outcome classes of these games determine $\gset(G)$.
  But, by assumption, we have $o(\game{\lambda} \dsum (G \osum \game{\rho})) = o(\game{\lambda} \dsum (H \osum \game{\rho}))$ for all $\lambda$ and $\rho$.
  Therefore we conclude that $\gset(G) = \gset(H)$.
\end{proof}

By the above theorem, the fundamental question

\smallskip
\centerline{\textit{What is the outcome of an ordered join of games?}}
\smallskip

\noindent necessitates the study of the Grundy sets of its components.
As far as ordered joins are concerned, the Grundy set is \emph{the} key invariant, just like the Grundy number is for disjunctive sums. 
It is thus a fundamental part of the presented results.
This motivation of the Grundy set, from first principles of game theory, also complements Theorem~\ref{thm:general_substitution}. 
It shows that the assumptions therein are not only sufficient but also necessary in general --- even if one is only interested in the outcome class of an ordered join of games.

However, by Theorem~\ref{thm:general_substitution} one still obtains the Grundy set of the compound, given its shape, from the Grundy sets of its components.
This is no coincidence, as it follows from the compositional properties of ordered joins, i.e.\ from Lemma~\ref{lem:composition}.
More precisely, these properties, together with Theorem~\ref{thm:outcomes}, imply a variant of Theorem~\ref{thm:general_substitution} with finitely many substitutions.

\section{Reduction Principles}\label{sec:reduction}

For the purpose of determining the Grundy number  of an ordered join $S \ojoin_i G_i$ of games, it is sometimes possible to entirely neglect certain components $G_i$.
In this section we explore two reduction principles of this kind.

\subsection{Zero Upper Sets}

Given a terminating game $S \ojoin_i G_i$, it follows from Theorem~\ref{thm:special_substitution} that a component $G_{i_0}$, with $G_{i_0} \simeq_0 \game{0}$ and $i_0$ maximal in $S$, does not contribute to the Grundy number of the compound.
By a trivial induction argument, any finite upper set $Z \subseteq S$ with $G_{i} \simeq_0 \game{0}$ for all $i \in Z$ can thus be neglected.
Our first reduction principle is an extension of this, which allows for infinite upper sets.

\begin{lemma}\label{lem:zero_upper_set}
  Let $S \ojoin_i G_i$ be a terminating game, and let $Z \subseteq S$ be an upper set.
  If $G_i \simeq_0 \game{0}$ for all $i \in Z$, then $S \ojoin_i G_i \simeq_0 (S \setminus Z) \ojoin_i G_i$.
\end{lemma}
\begin{proof}
  We write $G \coloneqq S \ojoin_i G_i$ and $H \coloneqq (S \setminus Z) \ojoin_i G_i$.
  Every option $H'$ of $H$ arises from a move in $i_0 \in S \setminus Z$.
  The same move also gives rise to an option $G'$ of $G$.
  By induction, $G'$ and $H'$ satisfy $\gnum(G') = \gnum(H')$.
  This proves $\gset(H) \subseteq \gset(G)$ and thus $\gnum(H) \leq \gnum(G)$.

  It suffices to show $\gnum(H) \neq \gnum(G')$ for each option $G'$ of $G$ obtained from a move in some $i_0 \in Z$.
  Since $G_{i_0} \simeq_0 \game{0}$ by assumption on $Z$, and thus $G'_{i_0} \not\simeq_0 \game{0}$, there exists a move $G'_{i_0} \to G''_{i_0}$ such that $G''_{i_0} \simeq_0 \game{0}$.
  The corresponding option $G''$ satisfies $\gnum(G'') = \gnum(H)$ by induction.
  Hence, we obtain $\gnum(H) = \gnum(G'') \neq \gnum(G')$.
\end{proof}

\subsection{Mirror Symmetry}

Our second reduction principle is a minor generalization of a result of Fenner and Rogers \cite[Lemma~2.21]{FR2015} and is concerned with ordered joins that exhibit a form of symmetry.
At the level of shapes, the precise setting is captured by the following. 

\begin{definition}
  Let $S$ be a poset.
  We call an involution $\sigma \colon S \to S$ \emph{weakly order-preserving} if the upper set $\set{j \in S \given i < j \text{ or } \sigma(i) < j}$ is $\sigma$-invariant for each $i\in S$. 
\end{definition}

Note that every order-preserving involution is also weakly order-preserving in the above sense.
The converse, however, is false.
For example, the poset $N \coloneqq \set{1 < 2 > 3 < 4}$ allows for three weakly order-preserving involutions, only one of which is order-preserving.

\medbreak

As a preliminary to our reduction principle, we prove its special case regarding fixed-point free involutions.

\begin{lemma}\label{lem:symmetry}
  Let $S \ojoin_i G_i$ be a terminating game, and let $\sigma \colon S \to S$ be a weakly order-preserving involution.
  If $G_i \simeq_1 G_{\sigma(i)}$ for all $i \in S$ and $\sigma$ has no fixed points, then $S \ojoin_i G_i \simeq_0 \game{0}$.
\end{lemma}

\begin{proof}
  Let $G \coloneqq S \ojoin_i G_i$.
  In view of Theorem~\ref{thm:general_substitution} we may assume, without loss of generality, that $G_i = G_{\sigma(i)}$ for all $i \in S$.

  Consider a move $G \to G'$ in $i_0$.
  Since the elements $i_0$ and $i_1 \coloneqq \sigma(i_0)$ of $S$ are incomparable, the component $G_{i_1}$ persists to $G'$, i.e.\ $G'_{i_1} = G_{i_1}$.
  Hence, the same move in $i_1$ results in an option $G''$ of $G'$ with $G''_{i_0} = G''_{i_1}$.
  By induction, we have $G'' \simeq_0 \game{0}$ and hence $\Gamma_0(G') \neq 0$.
  The latter is true for all $G'$, and therefore $\Gamma_0(G) = 0$, i.e.\ $G \simeq_0 \game{0}$.
\end{proof}

In the general case, we allow for fixed points of $\sigma$, provided that they form a lower set of the shape.
For the purpose of determining the Grundy number, we can then restrict ourselves to this set.
An example of this is depicted in Figure~\ref{fig:symmetry}.

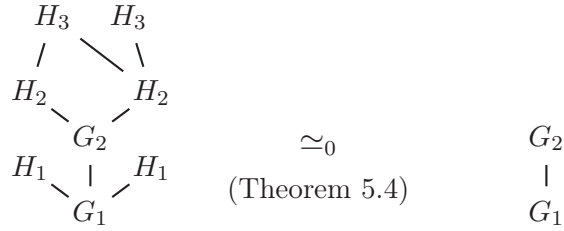
\begin{figure}[ht]
  \begin{tikzpicture}
    \tikzset{
        component/.style={circle, inner sep=0, minimum size=20pt},
        pics/symmetry 1/.style n args={5}{code={
          \node[component] (L0C) at (   0, -3.0) {\small #1};
          \node[component] (L1L) at (-2  , -1.5) {\small #3};
          \node[component] (L1C) at (   0, -0.5) {\small #2};
          \node[component] (L1R) at (+2  , -1.5) {\small #3};
          \node[component] (L2L) at (-2.0, +1.0) {\small #4};
          \node[component] (L2R) at (+2.0, +1.0) {\small #4};
          \node[component] (L3L) at (-1.25, +3.5) {\small #5};
          \node[component] (L3R) at (+1.25, +3.5) {\small #5};

          \draw[thick] (L0C) -- (L1L)  (L0C) -- (L1C)  (L0C) -- (L1R);
          \draw[thick] (L1C) -- (L2L)  (L1C) -- (L2R);
          \draw[thick] (L2L) -- (L3L);
          \draw[shorten <= 3pt, shorten >= 3pt, thick] (L2R) -- (L3L);
          \draw[thick] (L2R) -- (L3R);
        }},
        pics/symmetry 2/.style n args={2}{code={
          \node[component] (L0C) at (   0, -3.0) {\small #1};
          \node[component] (L1C) at (   0, -0.5) {\small #2};

          \draw[thick] (L0C) -- (L1C);
        }},
    }

    \pic[scale=0.4] (G)   at (-3,0) 
      {symmetry 1={$G_1$}{$G_2$}{$H_1$}{$H_2$}{$H_3$}};

    \node at (0, -.3) {$\simeq_0$};
    \node at (0, -.9) {\small(Theorem~\ref{thm:symmetry})};

    \pic[scale=0.4] (G')  at (+3,0) 
      {symmetry 2={$G_1$}{$G_2$}};
  \end{tikzpicture}
  \caption{Pruning a tulip using its symmetry.}
  \label{fig:symmetry}
\end{figure}

\begin{theorem}\label{thm:symmetry}
  Let $S \ojoin_i G_i$ be a terminating game and let $\sigma \colon S \to S$ be a weakly order-preserving involution.
  If $G_i \simeq_1 G_{\sigma(i)}$ for all $i \in S$ and the fixed points of $\sigma$ form a lower set $S^\sigma \subseteq S$, then $S \ojoin_i G_i \simeq_0 S^\sigma \ojoin_i G_i$.
\end{theorem}

\begin{proof}
  Consider the game $S \ojoin_i G_i \dsum S^\sigma \ojoin_i G_i$ as an ordered join in the shape $S + S^\sigma$.
  Further, let $\tau \colon S + S^\sigma \to S + S^\sigma$ be the map operating by $\sigma$ on $S \setminus S^\sigma$ and transposing the remaining two copies of $S^\sigma$ in $S + S^\sigma$.

  By construction, $\tau$ is a weakly order-preserving involution without fixed points.
  From Lemma~\ref{lem:symmetry}, we conclude $S \ojoin_i G_i \dsum S^\sigma \ojoin_i G_i \simeq_0 \game{0}$ and hence $S \ojoin_i G_i \simeq_0 S^\sigma \ojoin_i G_i$.
\end{proof}

\section{Computational Aspects}\label{sec:computation}

Throughout this article, some basic invariants of games have been essential for the presented theory --- the outcome class, the Grundy number, and the Grundy set.
For an ordered join, Theorem~\ref{thm:general_substitution} shows that these are completely determined by its shape and the Grundy sets of its components.
In this section we explore computational aspects of this relationship.

\medbreak

The Grundy set of short games, as well as the Grundy number and outcome class, can be evaluated na\"{i}vely.
This can be used to compute the Grundy set of an ordered join, even if only its shape $S$ and the Grundy sets of its components are given: one replaces the components by arbitrary games $G_i$ with the respective Grundy sets.
Herein, a natural choice is one resulting in a minimal number of subgames of $S \ojoin_i G_i$.
This is achieved, if each component itself has a minimal number of subgames, i.e.\ $G_i$ is the game with options $\game\lambda$, one for each $\lambda \in \gset(G_i)$.

This method of replacing components can also be useful in the general setting.
This approach has the potential to drastically reduce the number of subgames one has to consider in the na\"{i}ve algorithm, thereby offsetting the cost of determining suitable replacements.
Furthermore, the computations for each component are independent of one another.
One can thus use any available specialized algorithm for specific components and perform all these computations in parallel.

For disjunctive and ordinal sums, the basic invariants considered above can be computed with simple and efficient methods using the formulas discussed in Section~\ref{sec:preliminaries}.
In view of this, the above method for ordered joins seems unsatisfying at first.
However, even the simple case of computing the outcome class of $S \ojoin_i \game{1}$ is known to be a \PSPACE-complete problem \cite{G2013}, and thus the question at hand is a \PSPACE-hard problem.
Consequently, one should expect only minor improvements in this general setting.

One such improvement is obtained by leveraging the compositional structure of ordered joins with the help of Lemma~\ref{lem:composition}.
In order to formalize this, we need further notions and results from the theory of posets regarding the modular composition.

\begin{definition}
  A modular composition $S = R \mcomp_i S_i$ is called \emph{proper}, if each $S_i$ is nonempty and at least one $S_i$ satisfies $1 < \card{S_i} < \card{S}$.

  A poset $S$ is called \emph{decomposable}, if it can be represented by a proper modular composition.
  Otherwise it is called \emph{indecomposable}.
\end{definition}

\begin{theorem}[cf.\ {\cite[Satz 1]{G1967}}]\label{thm:poset_decomposition}
  Let $S$ be a finite decomposable poset.
  Then exactly one of the following statements is true.
  \begin{enumerate}
    \item\label{thm:poset_decomposition:c1} $S = R \mcomp_i S_i$ by a proper composition with $R$ indecomposable.
    \item\label{thm:poset_decomposition:c2} $S = R \mcomp_i S_i$ by a proper composition with $R$ an antichain.
    \item\label{thm:poset_decomposition:c3} $S = R \mcomp_i S_i$ by a proper composition with $R$ a chain.
  \end{enumerate}
\end{theorem}

In case (\ref{thm:poset_decomposition:c1}) the decomposition asserted by the above theorem is unique.
For the two other cases there exist unique maximal decompositions.
Recursively decomposing the posets $S_i$, one obtains a hierarchical representation of $S$.
It is unique, provided that maximal quotients are chosen along the way, and called the \emph{modular decomposition} of $S$.

For a short ordered join $S \ojoin_i G_i$, the modular decomposition of $S$ allows one to apply above replacement method in a recursive fashion.
Its potential improvements can thus be leveraged in multiple layers of the computation.
This approach is feasible, since the modular decomposition can be computed efficiently.\footnote{The modular decomposition of a poset agrees with the modular decomposition of its comparability graph (cf.\ \cite{G1967}).  For undirected graphs, various linear time algorithms are known (see e.g.\ \cite{TCHP2008} and references therein).}
The cases (\ref{thm:poset_decomposition:c2}) and (\ref{thm:poset_decomposition:c3}) of Theorem~\ref{thm:poset_decomposition} correspond to disjunctive and ordinal sums of games respectively.
As such, these cases allow for efficient computation using the appropriate formulas from Section~\ref{sec:preliminaries}.
In particular, if only those cases appear in the modular decomposition of $S$, i.e.\ if $S$ is a series-parallel poset, then $\gset(S \ojoin_i G_i)$ can be computed in polynomial time from the Grundy sets $\gset(G_i)$.
If $G_i = \game{1}$ for all components, this result is due to Deuber and Thomass\'{e} \cite{DT1996}.

The latter simplifications also motivate the following general question.

\begin{question}
  Given the Grundy sets of its components, for which indecomposable shapes $S$ can $\gset(S \ojoin_i G_i)$ be computed efficiently?
\end{question}

\section{Concluding Remarks}\label{sec:conclusion}

We have introduced the ordered join and given a complete analysis of it for impartial combinatorial games.
However, our definition of the ordered join applies equally well for partizan games.
In this case, Theorem~\ref{thm:outcomes} provides a possible definition of $1$-equivalence.
One can then derive a partizan variant of Theorem~\ref{thm:general_substitution} allowing for finitely many substitutions.
This is done in the same way as in the impartial case, which was discussed after Theorem~\ref{thm:outcomes}.

\begin{question}
  Which presented results generalize to partizan games?
\end{question}

In the theory of impartial games, the Sprague-Grundy Theorem shows the importance of $0$-equivalence for disjunctive sums --- studying the latter necessitates studying the former.
With Theorem~\ref{thm:outcomes}, we have shown that $1$-equivalence and ordered joins are related in the same way.

The notions of $0$- and $1$-equivalence can be seen as part of a hierarchy of equivalence relations:
$G$ and $H$ are $k$-equivalent, if $\gtype[k](G) = \gtype[k](H)$, where $\gtype[k](G) \coloneqq \set{\gtype[k-1](G') \given G \to G'}$.
Through induction, Theorem~\ref{thm:general_substitution} extends to higher orders of this hierarchy, i.e.\ $\gtype[k](S \ojoin_i G_i)$ is determined by $S$ and the $\gtype[k](G_i)$ whenever $k \geq 1$.
As such, this extension of Theorem~\ref{thm:general_substitution} is a compatibility condition between the hierarchy and ordered joins, and hence also disjunctive and ordinal sums.
However, these compounds do not necessitate studying the higher equivalences.

\begin{question}
  Is there an appropriate compound for $k$-equivalence?
\end{question}

\section*{Acknowledgements}
The authors would like to thank Prof.\ Dr.\ Michael Eisermann and Dr.\ Friederike Stoll for their helpful comments and suggestions.

\bigbreak

\end{document}